\documentclass[final]{siamltex}
\usepackage{geometry}
\usepackage[T1]{fontenc}
\usepackage{amsmath,amssymb,amsfonts,graphicx}
\usepackage{epsfig}
\usepackage{epstopdf}
\usepackage[parfill]{parskip}
\usepackage{cite}
\usepackage{listings}
\usepackage[usenames,dvipsnames]{xcolor}
\usepackage{hyperref}
\usepackage{todo}

\def\hdiv{{H(\mathrm{div})}}

\DeclareMathOperator{\Id}{Id}
\newcommand{\norm}[1]{\left\| #1 \right\|}
\newcommand{\seminorm}[1]{\left| #1 \right|}
\newcommand{\weightednorm}[2]{\norm{#1}_{#2}}
\newcommand{\weightedseminorm}[2]{\seminorm{#1}_{#2}}
\newcommand{\weightednormattime}[3]{\weightednorm{#1\left( \cdot , #3
    \right)}{#2}}
\newcommand{\normattime}[2]{\norm{#1\left( \cdot , #2 \right)}}
\newcommand{\weightedseminormattime}[3]{\weightedseminorm{#1\left( \cdot , #3 \right)}{#2}}

\title{Mixed finite elements for global tide models}
\author{ Colin Cotter \thanks{Department of Mathematics, Imperial College London, South Kensington Campus, London SW7 2AZ. This work was supported by NERC grant NE/I016007/1}
\and
Robert C. Kirby  \thanks{Department of Mathematics, Baylor University, One Bear Place \#97328, Waco, TX 76798-73278.  (\url{robert_kirby@baylor.edu}).  This work was supported by NSF grant CCF-1117794.}
}

\begin{document}
\maketitle

\pagestyle{myheadings}
\thispagestyle{plain}
\markboth{Cotter, Kirby}{MFEM for Tides} 

\renewcommand{\thefootnote}{\fnsymbol{footnote}}

\begin{abstract}
We study mixed finite element methods for the linearized rotating
shallow water equations with linear drag and forcing terms.
By means of a strong energy estimate for an equivalent second-order
formulation for the linearized momentum, we prove long-time stability
of the system without energy accumulation --  the geotryptic state.
\emph{A priori} error estimates for the linearized momentum and free
surface elevation are given in $L^2$ as well as for the time
derivative and divergence of the linearized momentum.  Numerical
results confirm the theoretical results regarding both energy damping
and convergence rates.
\end{abstract}




\section{Introduction}

Finite element methods are attractive for modelling the world's oceans
since implemention with triangular cells provides a
means to accurately represent coastlines and topography
~\cite{We_etal2010}. In the last decade or
so, there has been much discussion about the best choice of mixed
finite element pairs to use as the horizontal discretization for
atmosphere and ocean models. In particular, much attention has been
paid to the properties of numerical dispersion relations obtained when
discretizing the rotating shallow water equations
\cite{Da2010,CoHa2011,CoLaReLe2010,RoRoPo2007,Ro2005,RoBe2009,RoRo2008,Ro2012}.
In this paper we take a different angle, and study the behavior of
discretizations of forced-dissipative rotating shallow-water
equations, which are used for predicting global barotropic tides. The
main point of interest here is whether the discrete solutions approach
the correct long-time solution in response to quasi-periodic forcing.
In particular, we study the behavior of the linearized energy. Since
this energy only controls the divergent part of the solution, as we
shall see later, it is important to choose finite element spaces
where there is a natural discrete Helmholtz decomposition, and where
the Coriolis term projects the divergent and divergence-free
components of vector fields correctly onto each other. Hence, we
choose to concentrate on the mimetic, or compatible, finite element
spaces (\emph{i.e.} those which arise naturally from the finite
element exterior calculus \cite{arnold2006finite}) which were proposed
for numerical weather prediction in \cite{CoSh2012}. In that paper, it
was shown that the discrete equations have an exactly steady
geostrophic state (a solution in which the Coriolis term balances the
pressure gradient) corresponding to each of the divergence-free
velocity fields in the finite element space; this approach was
extended to develop finite element methods for the nonlinear rotating
shallow-water equations on the sphere that can conserve energy,
enstrophy and potential vorticity
\cite{RoHaCoMc2013,CoTh2014,McCo2014}. Here, we shall make use of the
discrete Helmholtz decomposition in order to show that mixed finite
element discretizations of the forced-dissipative linear rotating
shallow-water equations have the correct long-time energy
behavior. Since we are studying linear equations, these energy
estimates then provide finite time error bounds.

Predicting past and present ocean tides is important because they have
a strong impact on sediment transport and coastal flooding, and hence
are of interest to geologists. Recently, tides have also received a
lot of attention from global oceanographers since breaking internal
tides provide a mechanism for vertical mixing of temperature and
salinity that might sustain the global ocean circulation
\cite{MuWu1998,GaKu2007}. A useful tool for predicting tides are the
rotating shallow water equations, which provide a model of the
barotropic (\emph{i.e.}, depth-averaged) dynamics of the ocean. When
modelling global barotropic tides away from coastlines, the nonlinear
advection terms are very weak compared to the drag force, and a
standard approach is to solve the linear rotating shallow-water
equations with a parameterised drag term to model the effects of
bottom friction, as described in \cite{La45}. This approach can be
used on a global scale to set the boundary conditions for a more
complex regional scale model, as was done in \cite{Hi_etal2011}, for
example. Various additional dissipative terms have been proposed to
account for other dissipative mechanisms in the barotropic tide, due
to baroclinic tides, for example \cite{JaSt2001}. 

As mentioned above, finite element methods provide useful
discretizations for tidal models since they can be used on
unstructured grids which can seamless couple global tide structure
with local coastal dynamics. A discontinuous Galerkin approach was
developed in \cite{salehipour2013higher}, whilst continuous finite
element approaches have been used in many studies \cite[for example]{kawahara1978periodic,Le_etal2002,FoHeWaBa1993}. The lowest order Raviart-Thomas element for
velocity combined with $P_0$ for height was proposed for coastal tidal
modeling in \cite{walters2005coastal}; this pair fits into the framework 
that we discuss in this paper.

In this paper we will restrict attention to the linear bottom drag
model as originally proposed in \cite{La45}. We are aware that the
quadratic law is more realistic, but the linear law is more amenable
to analysis and we believe that the correct energy behavior of
numerical methods in this linear setting already rules out many
methods which are unable to correctly represent the long-time solution
which is in geotryptic balance (the extension to geostrophic balance
of the three way balance between Coriolis, the pressure gradient and
the dissipative term). In the presence of quasiperiodic time-varying
tidal forcing, the equations have a time-varying attracting solution
that all solutions converge to as $t\to \infty$. In view of this, we
prove the following results which are useful to tidal modellers (at
least, for the linear law):
\begin{enumerate}
\item For the mixed finite element methods that we consider, the
  spatial semidiscretization also has an attracting solution 
  in the presence of time-varying forcing.
\item This attracting solution converges to time-varying attracting
  solution of the unapproximated equations.
\end{enumerate}

Global problems require tidal simulation on manifolds rather than
planar domains.  For simplicity, our description and analysis will
follow the latter case.  However, our numerical results include
the former case.  Recently, Holst and Stern~\cite{holst2012geometric}
have demonstrated that finite element analysis on discretized
manifolds can be handled as a variational crime.  We summarize these
findings and include an appendix at the end demonstrating how to apply
their techniques to our own case.  This suggests that the extension to
manifolds presents technicalities rather than difficulties to the
analysis we provide here.

The rest of this paper is organised as follows. In Section
\ref{se:model} we describe the finite element modelling framework
which we will analyse. In Section \ref{se:prelim} we provide
some mathematical preliminaries. In Section \ref{se:energy} we derive
energy stability estimates for the finite element tidal equations. In 
Section \ref{se:error} we use these energy estimates to obtain error
bounds for our numerical solution.  Appendix~\ref{ap:bendy} includes
the discussion of embedded manifolds.

\section{Description of finite element tidal model}
\label{se:model}
We start with the nondimensional linearized rotating shallow water
model with linear drag and forcing on a (possibly curved) two
dimensional surface $\Omega$, given by
\begin{equation}
\begin{split}
u_t + \frac{f}{\epsilon} u^\perp + \frac{\beta}{\epsilon^2} \nabla
\left( \eta - \eta^\prime \right) + C u & = 0, \\
\eta_t + \nabla \cdot \left( H u \right) & = 0,
\end{split}
\end{equation}
where $u$ is the nondimensional two dimensional velocity field tangent
to $\Omega$, $\eta$ is the nondimensional free surface elevation above
the height at state of rest, $\nabla\eta'$ is the (spatially varying)
tidal forcing, $\epsilon$ is the Rossby number (which is small for
global tides), $f$ is the spatially-dependent non-dimensional Coriolis
parameter which is equal to the sine of the latitude (or which can be
approximated by a linear or constant profile for local area models),
$\beta$ is the Burger number (which is also small), $C$ is the (spatially varying) nondimensional drag
coefficient and $H$ is the (spatially varying) nondimensional fluid
depth at rest, and $\nabla$ and $\nabla\cdot$ are the intrinsic
gradient and divergence operators on the surface $\Omega$,
respectively.

We will work with a slightly generalized version of the forcing term,
which will be necessary for our later error analysis.  Instead of
assuming forcing of the form 
$\frac{\beta}{\epsilon^2} \nabla \eta^\prime$, 
we assume some $F \in L^2$, giving our model as
\begin{equation}
\begin{split}
u_t + \frac{f}{\epsilon} u^\perp + \frac{\beta}{\epsilon^2} \nabla
 \eta  + C u & = F, \\
\eta_t + \nabla \cdot \left( H u \right) & = 0.
\end{split}
\end{equation}

It also becomes useful to work in terms of the linearized momentum
$\widetilde{u} = H u$ rather than velocity. After making this substitution and dropping the
tildes, we obtain
\begin{equation}
\begin{split}
\frac{1}{H}u_t + \frac{f}{H\epsilon} u^\perp + \frac{\beta}{\epsilon^2} \nabla
\eta  + \frac{C}{H} u & = F, \\
\eta_t + \nabla \cdot  u & = 0.
\end{split}
\label{eq:thepde}
\end{equation}
A natural weak formulation of this equations is to seek $u \in \hdiv$
and $\eta \in L^2$ so that
\begin{equation}
\begin{split}
\left( \frac{1}{H}u_t , v \right) 
+ \frac{1}{\epsilon} \left( \frac{f}{H} u^\perp , v \right) 
- \frac{\beta}{\epsilon^2} \left( \eta ,
\nabla \cdot v \right) + \left( \frac{C}{H} u , v \right) & = 
\left( F , v \right)
, \quad \forall v \in \hdiv, \\
\left( \eta_t , w \right) + \left( \nabla \cdot u  , w \right)& = 0, \quad
\forall w \in L^2.
\end{split}
\label{eq:mixed}
\end{equation}

We now develop mixed discretizations with $V_h \subset \hdiv$ and $W_h
\subset L^2$.  Conditions on the spaces are the commuting projection
and divergence mapping $V_h$ onto $W_h$. We define $u_h \subset V_h$
and $\eta_h \subset W_h$ as solutions of the discrete variational
problem
\begin{equation}
\begin{split}
\left( \frac{1}{H}u_{h,t} , v_h \right) 
+ \frac{1}{\epsilon} \left( \frac{f}{H} u_h^\perp , v_h \right) 
- \frac{\beta}{\epsilon^2} \left( \eta_h ,
\nabla \cdot v_h \right) + \left( \frac{C}{H} u_h , v_h \right) & =
\left( F , v_h \right)
, \\
\left( \eta_{h,t} , w_h \right) + \left( \nabla \cdot u_h  , w_h \right)& = 0.
\end{split}
\label{eq:discrete_mixed}
\end{equation}

We will eventually obtain stronger estimates by working with an
equivalent second-order form.  
If we take the time derivative of the
first equation in~(\ref{eq:discrete_mixed}) and use the fact that
$\nabla \cdot V_h = W_h$, we have
\begin{equation}
\left( \frac{1}{H}u_{h,tt} , v_h \right) 
+ \frac{1}{\epsilon} \left( \frac{f}{H} u_{h,t}^\perp , v_h \right) 
+ \frac{\beta}{\epsilon^2} \left( \nabla \cdot u_h , \nabla \cdot v_h \right)
+ \left( \frac{C}{H} u_{h,t} , v_h \right)  = 
\left( \widetilde{F} ,
v_h \right),
\label{eq:secondorderdiscrete}
\end{equation}
where $\widetilde{F} = F_t$.  This is a restriction of
\begin{equation}
\left( \frac{1}{H}u_{tt} , v \right) 
+ \frac{1}{\epsilon} \left( \frac{f}{H} u_{t}^\perp , v \right) 
+ \frac{\beta}{\epsilon^2} \left( \nabla \cdot u , \nabla \cdot v \right)
+ \left( \frac{C}{H} u_t , v \right)  = 
\left( \widetilde{F} ,
v_h \right),
\end{equation}
which is the variational form of
\begin{equation}
\frac{1}{H} u_{tt} + \frac{f}{H} u^\perp_t 
- \frac{\beta}{\epsilon^2} \nabla \left( \nabla
\cdot u \right)
+ \frac{C}{H} u_t = \widetilde{F},
\end{equation}
to the mixed finite element spaces.

We have already discussed mixed finite elements' application to tidal
models in the geophysical literature, but this work also builds on 
existing literature for mixed discretization of the acoustic
equations. The first such investigation is due to
Geveci~\cite{geveci1988application}, where exact energy conservation
and optimal error estimates are given for the semidiscrete first-order
form of the model wave equation.  Later
analysis~\cite{cowsar1990priori,jenkins2003priori} considers a second 
order in time wave equation with an auxillary flux at each time step.
In~\cite{kirbykieu}, Kirby and Kieu return to the first-order
formulation, giving additional estimates
beyond~\cite{geveci1988application} and also analyzing the symplectic
Euler method for time discretization.  From the standpoint of this
literature, our model~(\ref{eq:thepde}) appends additional terms for
the Coriolis force and damping to the simple acoustic model.
We restrict ourselves to semidiscrete analysis in this work, but pay
careful attention the extra terms in our estimates, showing how study
of an equivalent second-order equation in $\hdiv$ proves proper
long-term behavior of the model.

\section{Mathematical preliminaries}

\label{se:prelim}

For the velocity space $V_h$, we will work with standard $\hdiv$ mixed
finite element spaces on triangular elements, such as Raviart-Thomas
(RT), Brezzi-Douglas-Marini (BDM), and Brezzi-Douglas-Fortin-Marini
(BDFM)~\cite{RavTho77a,brezzi1985two,brezzi1991mixed}.  We label the
lowest-order Raviart-Thomas space with index $k=1$, following the
ordering used in the finite element exterior
calculus~\cite{arnold2006finite}.  Similarly, the lowest-order
Brezzi-Douglas-Fortin-Marini and Brezzi-Douglas-Marini spaces
correspond to $k=1$ as well.  We will always take $W_h$ to consist of
piecewise polynomials of degree $k-1$, not constrained to be
continuous between cells. In the case of domains with boundaries, we
require the strong boundary condition $u\cdot n = 0$ on all
boundaries. 

In the main part of this paper we shall present results assuming that
the domain is a subset of $\mathbb{R}^2$, \emph{i.e.} flat
geometry. In the Appendix, we describe how to extend these results to
the case of embedded surfaces in $\mathbb{R}^3$.

Throughout, we shall let $\norm{\cdot}$ denote the standard $L^2$ norm.  We
will frequently work with weighted $L^2$ norms as well.  For a
positive-valued weight function $\kappa$, we define the weighted 
$L^2$ norm 
\begin{equation}
\weightednorm{g}{\kappa}^2
= \int_\Omega \kappa \left| g \right|^2 dx.
\end{equation}
If there exist positive constants $\kappa_*$ and $\kappa^*$ such that
$0 < \kappa_* \leq \kappa \leq \kappa^* < \infty$ 
almost everywhere, then the weighted norm is equivalent to
the standard $L^2$ norm by
\begin{equation}
\sqrt{\kappa_*} \norm{g} 
\leq \weightednorm{g}{\kappa} 
\leq \sqrt{\kappa^*} \norm{g}.
\end{equation}

A Cauchy-Schwarz inequality 
\begin{equation}
(\kappa g_1 , g_2) \leq 
\weightednorm{g_1}{\kappa}
\weightednorm{g_2}{\kappa}
\end{equation}
holds for the weighted inner product, and we can also incorporate
weights into Cauchy-Schwarz for the standard $L^2$ inner product by
\begin{equation}
(g_1,g_2) = 
(\sqrt{\kappa} g_1 , \frac{1}{\sqrt{\kappa}} g_2)
\leq \weightednorm{g_1}{\kappa} \weightednorm{g_2}{\frac{1}{\kappa}}.
\end{equation}

We refer the reader to references such as~\cite{brezzi1991mixed} for full
details about the particular definitions and properties of these
spaces, but here recall several facts essential for our analysis.  For
all velocity spaces $V_h$ we consider, the divergence maps $V_h$ onto $W_h$.
Also, the spaces of interest all have a projection, $\Pi : \hdiv
\rightarrow V_h$ that commutes with the $L^2$ projection $\pi$ into
$W_h$: 
\begin{equation}
\left( \nabla \cdot \Pi u , w_h \right)
= \left( \pi \nabla \cdot u , w_h \right)
\end{equation}
for all $w_h \in W_h$ and any $u \in \hdiv$.  
We have the error estimate
\begin{equation}
\norm{u - \Pi u} \leq C_{\Pi} h^{k+\sigma} \weightedseminorm{u}{k}
\label{eq:PiL2}
\end{equation}
when $u \in H^{k+1}$.  Here, $\sigma = 1$ for the BDM spaces but
$\sigma =0$ for the RT or BDFM spaces. The projection also has an
error estimate for the divergence
\begin{equation}
\norm{ \nabla \cdot \left( u - \Pi u \right) } \leq C_\Pi h^{k} 
\weightedseminorm{\nabla \cdot u}{k}
\label{eq:PiDiv}
\end{equation}
for all the spaces of interest, whilst the pressure projection has the
error estimate 
\begin{equation}
\norm{ \eta - \pi \eta } \leq C_{\pi} h^k \weightedseminorm{\eta}{k}.
\label{eq:piL2}
\end{equation}
Here, $C_\Pi$ and $C_\pi$ are positive constants independent of $u$,
$\eta$, and $h$, although not necessarily of the shapes of the
elements in the mesh.

We will utilize a Helmholtz decomposition of $\hdiv$ under a weighted inner product.  For a very general treatment of such decompositions,
we refer the reader to~\cite{arnold2010finite}.  For each $u \in V$, there exist
unique vectors $u^D$ and $u^S$ such that $u = u^D + u^S$, $\nabla
\cdot u^S = 0$, and also $\left( \frac{1}{H} u^D , u^S \right) = 0$.
That is, $\hdiv$ is decomposed into the direct sum of solenoidal
vectors, which we denote by
\begin{equation}
\mathcal{N} \left( \nabla \cdot \right)
= \left\{ u \in V : \nabla \cdot u = 0 \right\},
\end{equation}
 and its orthogonal complement under the $\left( \frac{1}{H}
  \cdot , \cdot \right)$ inner product, which we denote by
\begin{equation}
\mathcal{N} \left( \nabla \cdot \right)^\perp
= \left\{ u \in V : \left( \frac{1}{H} u , v \right) = 0, \ \forall v
  \in \mathcal{N} \left( \nabla \cdot \right) \right\}.
\end{equation}
Functions in $\mathcal{N} \left( \nabla \cdot \right)^\perp$
satisfy a generalized Poincar\'e-Friedrichs inequality, that there
exists some $C_P$ such that 
\begin{equation}
\weightednorm{u^D}{\frac{1}{H}}
\leq C_P \weightednorm{\nabla \cdot u^D}{\frac{1}{H}}.
\label{eq:pf}
\end{equation}
We may also use norm equivalence to write this as
\begin{equation}
\weightednorm{u^D}{\frac{1}{H}}
\leq \frac{C_P}{\sqrt{H_*}} \norm{\nabla \cdot u^D}.
\end{equation}
Because our mixed spaces $V_h$ are contained in $H(\mathrm{div})$, the same decompositions can be applied, and the Poincar\'e-Friedrichs inequality holds with a constant no larger than $C_p$.

\section{Energy estimates}
\label{se:energy}
In this section, we develop in stability estimates for our system, obtained by energy techniques.
Supposing that there is no forcing
or damping ($F = C = 0$), we pick $v_h = u_h$ and 
$w_h =\frac{\beta}{\epsilon^2} \eta_h$ 
in~(\ref{eq:discrete_mixed}), and find that
\begin{equation}
\begin{split}
\left( \frac{1}{H} u_{h,t} , u_h \right)
+ \frac{1}{\epsilon} \left( \frac{f}{H} u_h^\perp , u_h \right)
- \frac{\beta}{\epsilon^2}\left( \eta_h , \nabla \cdot u_h \right) & = 0, \\
\frac{\beta}{\epsilon^2} \left( \eta_{h,t} , \eta_h \right)
+ \frac{\beta}{\epsilon^2} \left( \nabla \cdot u_h , \eta_h \right) &
= 0.
\end{split}
\end{equation}
Since $u_h^\perp \cdot u_h = 0$ pointwise, we add these two
equations together to find
\begin{equation}
\frac{1}{2} \frac{d}{dt} \weightednorm{u_h}{\frac{1}{H}}^2 
+ \frac{\beta}{2\epsilon^2} \frac{d}{dt} \norm{\eta_h}^2 = 0.
\end{equation}
Hence, we have the following.
\begin{proposition}
In the absence of damping or forcing, the quantity
\begin{equation}
E_1(t) = \frac{1}{2} \weightednorm{ u_h}{\frac{1}{H}}^2 
+ \frac{\beta}{2\epsilon^2} \norm{\eta_h}^2 
\label{eq:firstorderenergy}
\end{equation}
is conserved exactly for all time.
\label{prop:cons}
\end{proposition}

Now suppose that $F = 0$ still but that $0 < C_* \leq C \leq C <
\infty$ pointwise in $\Omega$.  The same considerations now lead
to
\begin{equation}
\frac{1}{2} \frac{d}{dt} \weightednorm{u_h}{\frac{1}{H}}^2
+ \frac{\beta}{2\epsilon^2} \frac{d}{dt} \norm{\eta_h}^2 
+ \weightednorm{u_h}{\frac{C}{H}}^2 = 0,
\end{equation}
so that
\begin{proposition}
In the absence of forcing, but with $0 < C_* \leq C \leq C < \infty$,
the quantity $E_1(t)$ defined in~(\ref{eq:firstorderenergy})
satisfies
\[
\frac{d}{dt} E_1(t) \leq 0.
\]
\label{prop:monotonic}
\end{proposition}

In the presence of forcing and dissipation, it is also possible to
make estimates showing worst-case linear accumulation of the energy
over time. 
\begin{proposition}
\label{prop:firstorderstability}
With nonzero $F$, we have that for all time $t$,
\begin{equation}
E_1(t) \leq E_1(0) + \frac{1}{2C_*} \int_0^t \weightednormattime{F}{H}{s}^2 ds
\end{equation}
\end{proposition}
\begin{proof}
We choose $w_h$ and $v_h$ as without forcing, and find that
\[
\frac{d}{dt} E_1(t) + \weightednormattime{u}{\frac{C}{H}}{t}^2 
= \left( F , u_h \right).
\]
Cauchy-Schwarz, Young's inequality, and norm equivalence give
\[
\frac{d}{dt} E_1(t)
+ \frac{C_*}{2} \weightednormattime{u_h}{\frac{1}{H}}{t}^2
\leq
\frac{1}{2C_*} \weightednormattime{F}{H}{t}^2 
\]
The result follows by dropping the positive term from the left-hand side and integrating.
\end{proof}

However, linear energy accumulation is not observed for actual tidal
motion, so we expect a stronger result to hold.  Turning to the second
order equation~(\ref{eq:secondorderdiscrete}), we 
begin with vanishing forcing and damping terms, putting $v_h = u_{h,t}$
to find
\begin{equation}
\left( \frac{1}{H}u_{h,tt} , u_{h,t} \right) 
+ \frac{1}{\epsilon} \left( \frac{f}{H} u_{h,t}^\perp , u_{h,t} \right) 
+ \frac{\beta}{\epsilon^2} \left( \nabla \cdot u_h , \nabla \cdot u_{h,t} \right)
= 0,
\end{equation}
which simplifies to
\begin{equation}
\frac{1}{2} \frac{d}{dt} \weightednorm{ u_{h,t}}{\frac{1}{H}}^2
+ \frac{\beta}{2\epsilon^2} \frac{d}{dt} \norm{\nabla \cdot u_h}^2  = 0,
\end{equation}
so that the quantity
\begin{equation}
E(t) = \frac{1}{2} \weightednorm{ u_{h,t}}{\frac{1}{H}}^2
+ \frac{\beta}{2\epsilon^2} \norm{\nabla \cdot u_h}^2
\label{eq:Edef}
\end{equation}
is conserved exactly for all time.

If $C$ is nonzero, we have that
\begin{equation}
\frac{1}{2} \frac{d}{dt} \weightednorm{ u_{h,t}}{\frac{1}{H}}^2
+ \frac{\beta}{2\epsilon^2} \frac{d}{dt} \norm{\nabla \cdot u_h}^2  
+ \weightednorm{u_{h,t}}{\frac{C}{H}}^2 = 0,
\end{equation}
which implies that $E(t)$ is nonincreasing, although with no
particular decay rate. 

Now, we develop more refined technique based on the Helmholtz decomposition
that gives a much stronger damping result.  We can write
$u_h = u_h^D + u_h^S$ in the $\frac{1}{H}$-weighted decomposition.
We let $0 < \alpha$ be a scalar to be determined later and let
the test function $v$ in~(\ref{eq:secondorderdiscrete}) be
$v_h = u_{h,t} + \alpha u^D_h$.  This gives
\begin{equation}
\begin{split}
\left( \frac{1}{H}u_{h,tt} ,  u_{h,t} + \alpha u^D_h \right) 
+ \frac{1}{\epsilon} \left( \frac{f}{H} u_{h,t}^\perp ,  u_{h,t} + \alpha u^D_h \right) & \\
+ \frac{\beta}{\epsilon^2} \left( \nabla \cdot u_h , \nabla \cdot
  \left(  u_{h,t} + \alpha u^D_h \right) \right)
+ \left( \frac{C}{H} u_{h,t} ,  u_{h,t} + \alpha u^D_h \right)  & = 0,
\end{split}
\end{equation}
and we rewrite the left-hand side so that
\begin{equation}
\begin{split}
\frac{1}{2} \frac{d}{dt} \weightednorm{ u_{h,t} }{\frac{1}{H}}^2
+ \alpha \left( \frac{1}{H} u_{h,tt} , u_h^D \right)
+ \frac{\alpha}{\epsilon} \left( \frac{f}{H} u_{h,t}^\perp , u_h^D \right) & \\
+ \frac{\beta}{2 \epsilon^2} \frac{d}{dt} \norm{\nabla \cdot u_h^D}^2
+ \frac{\alpha \beta}{\epsilon^2} \norm{ \nabla \cdot u_h^D }^2
+ \weightednorm{ u_{h,t} }{\frac{C}{H}}^2 
+ \alpha \left( \frac{C}{H}  u_{h,t} , u_{h}^D \right) & = 0.
\end{split}
\end{equation}
We use the fact that
\[
\frac{d}{dt} \left( \frac{1}{H} u_{h,t} , u_h^D \right)
= \left( \frac{1}{H} u_{h,tt} , u_h^D \right)
+ \left( \frac{1}{H} u_{h,t} , u_{h,t}^d \right)
\]
and also that $u_h^S$ is $\frac{1}{H}$-orthogonal to $u_h^D$
to rewrite the left-hand side as
\begin{equation}
\begin{split}
\frac{d}{dt}
\left[
\frac{1}{2} \weightednorm{u_{h,t}}{\frac{1}{H}}^2
+ \alpha \left( \frac{1}{H} u_{h,t} , u_h^D \right)
+ \frac{\beta}{2 \epsilon^2} \norm{ \nabla \cdot u_h^D }^2
\right] & \\
+ \frac{\alpha}{\epsilon} \left( \frac{f}{H} u_{h,t}^\perp , u_h^D \right)
+ \frac{\alpha\beta}{\epsilon^2} \norm{ \nabla \cdot u_h^D }^2 & \\
+ \weightednorm{ u_{h,t} }{\frac{C}{H}}^2 
- \alpha \weightednorm{ u_{h,t}^D }{\frac{1}{H}}^2 
+ \alpha \left( \frac{C}{H} u_{h,t} , u_h^D \right) & = 0.
\end{split}
\end{equation}
This has the form of an ordinary differential equation
\begin{equation}
A^\prime(t) + B(t) = 0,
\label{eq:ode}
\end{equation}
where
\begin{equation}
A(t) = \frac{1}{2} \weightednorm{u_{h,t}}{\frac{1}{H}}^2
+ \alpha \left( \frac{1}{H} u_{h,t} , u_h^D \right)
+ \frac{\beta}{2 \epsilon^2} \norm{ \nabla \cdot u_h^D }^2
\end{equation}
and
\begin{equation}
\begin{split}
B(t) = &
\frac{\alpha}{\epsilon} \left( \frac{f}{H} u_{h,t}^\perp , u_h^D \right)
+ \frac{\alpha\beta}{\epsilon^2} \norm{ \nabla \cdot u_h^D }^2 \\
& + \weightednorm{ u_{h,t} }{\frac{C}{H}}^2 
- \alpha \weightednorm{ u_{h,t}^D }{\frac{1}{H}}^2 
+ \alpha \left( \frac{C}{H} u_{h,t} , u_h^D \right).
\end{split}
\end{equation}
By showing that for suitably chosen $\alpha$, both $A(t)$ and $B(t)$ are comparable to $E(t)$ defined in~(\ref{eq:Edef}), we can obtain exponential damping of the energy.

\begin{lemma}
Suppose that 
\begin{equation}
\alpha \leq \alpha_1 \equiv \frac{\sqrt{\beta H_*}}{2 C_p \epsilon}.
\label{eq:alpha1}
\end{equation}
Then
\begin{equation}
\frac{1}{2} E(t) \leq A(t) \leq \frac{3}{2} E(t).
\label{eq:AequivE}
\end{equation}
\end{lemma}
\begin{proof}
We bound the term $\left( \frac{1}{H} u_{h,t} , u_h^D \right)$, with
Cauchy-Schwarz, Poincare-Friedrichs~(\ref{eq:pf}), and weighted
Young's inequality with $\delta = \frac{\epsilon}{\sqrt{\beta}}$: 
\begin{equation}
\begin{split}
\left( \frac{1}{H} u_{h,t}^D , u_h^D \right) &
\leq \frac{C_P}{2\sqrt{H_*}} \left[
\frac{\epsilon}{\sqrt{\beta}} \weightednorm{u_{h,t}}{\frac{1}{H}}^2
+ \frac{\sqrt\beta}{\epsilon} \norm{\nabla \cdot u_h^D}^2 \right] \\
& = 
\frac{C_P\epsilon}{\sqrt{H_*\beta}} \left[
\frac{1}{2} \weightednorm{u_{h,t}}{\frac{1}{H}}^2
+ \frac{\beta}{2\epsilon^2} \norm{\nabla \cdot u_h^D}^2 \right] \\
& = \frac{C_P\epsilon}{\sqrt{H_*\beta}} E(t).
\end{split}
\end{equation}
So, then, we have
\begin{equation}
\begin{split}
A(t) & \leq \left( 1 + \frac{\alpha C_P \epsilon}{\sqrt{\beta H_*}} \right)
E(t), \\
A(t) & \geq \left( 1 - \frac{\alpha C_P \epsilon}{\sqrt{\beta H_*}} \right)
E(t),
\end{split}
\end{equation}
and the result follows thanks to the assumption~(\ref{eq:alpha1}).
\end{proof}

Showing that $B(t)$ is bounded above by a constant times $E(t)$ is
straightforward, but not needed for our damping results.
\begin{lemma}
Suppose that 
\begin{equation}
0 < \alpha \leq \alpha_2 \equiv \frac{2 C_*}{1 + \chi},
\label{eq:alpha2}
\end{equation}
where
\begin{equation}
\chi =  \left( 2 +  \frac{C_P^2 \left( 1 + \epsilon C^*
      \right)^2}{ \beta H_*} \right).
\label{eq:chidef}
\end{equation}
Then
\begin{equation}
B(t) \geq \alpha E(t).
\label{eq:Bbelow}
\end{equation}
\end{lemma}
\begin{proof}
We use Cauchy Schwarz, the bounds $0 < C_* \leq C \leq C^*$ and $|f| \leq 1$, and
Young's inequality with weight $\delta > 0$ to write
\begin{equation}
\begin{split}
B(t) \geq &
\left( C_* - \alpha \right)
\weightednorm{u_{h,t}}{\frac{1}{H}}^2 
+ \frac{\alpha \beta}{\epsilon^2} \norm{\nabla \cdot u_h^D}^2 \\
&
- \frac{\alpha C_P}{\epsilon \sqrt{H_*}} \left( C^* \epsilon + 1 \right)
\weightednorm{u_{h,t}}{\frac{1}{H}}
  \norm{\nabla \cdot u_h^D} \\
 \geq &
\left[ 2 C_* - \alpha \left( 2 + \frac{C_P \left( 1 + \epsilon C^*
      \right)}{\epsilon \sqrt{H_*} \delta} \right) \right]
\frac{1}{2} \weightednorm{u_{h,t}}{\frac{1}{H}}^2 \\
&
+
\alpha
\left[ 2 - \frac{\epsilon C_P \left( 1 + \epsilon C^* \right)}{\beta
    \sqrt{H_*}} \delta \right] \frac{\beta}{2\epsilon^2} \norm{\nabla
  \cdot u_h^D}^2.
\end{split}
\end{equation}
Next, it remains to select $\delta$ and $\alpha$ to make the
coefficients of each norm positive and also balance the terms.  
First, we pick
\[
\delta = \frac{\beta \sqrt{H_*}}{\epsilon C_P \left( 1 + \epsilon C^* \right)},
\]
and calculating that
\[
\frac{C_P \left( 1 + \epsilon C^*
      \right)}{\epsilon \sqrt{H_*} \delta}
= \frac{C_P^2 \left( 1 + \epsilon C^* \right)^2}{ \beta H_*},
\]
we have that
\begin{equation}
\begin{split}
B(t) \geq &
\left[ 2 C_* - \alpha \left( 2 +  \frac{C_P^2 \left( 1 + \epsilon C^*
      \right)^2}{ \beta H_*} \right) \right] 
\frac{1}{2} \weightednorm{u_{h,t}}{\frac{1}{H}}^2
+ \alpha \frac{\beta}{2\epsilon^2} \norm{ \nabla \cdot u_h^D}^2 \\
& = \left( 2 C_* - \alpha \chi \right) 
    \frac{1}{2} \weightednorm{u_{h,t}}{\frac{1}{H}}^2
    + \frac{\alpha\beta}{2\epsilon^2} \norm{\nabla \cdot u_h^D}^2.
\end{split}
\label{eq:Blower2}
\end{equation}
We let $\alpha_2$ be the solution to
\[
2 c_* - \alpha_2 \chi = \alpha_2,
\]
so that
\begin{equation}
\alpha_2 \equiv \frac{2 C_*}{1 + \chi}.
\label{eq:alpha2choice}
\end{equation}
If we pick $\alpha = \alpha_2$, then we have the lower bound for 
$B(t)$ is exactly $\alpha E(t)$.   However, we are also constrained to
pick $\alpha \leq \min\{\alpha_1,\alpha_2\}$ in order
to guarantee that the lower bounds for $A(t)$ is positive as well.
If we have $\alpha \leq \alpha_2$, then
\[
2 C_* - \alpha \chi \geq 2 C_* - \alpha_2 \chi
= \alpha_2 \geq \alpha,
\]
and so we also have
\begin{equation}
B(t) \geq \alpha E(t).
\end{equation}
\end{proof}

We combine these two propositions to give our exponential damping
result.
\begin{theorem}
\label{th:exp}
Let $\alpha_1$ and $\alpha_2$ be defined by~(\ref{eq:alpha1})
and~(\ref{eq:alpha2}), respectively.  Then, for any
$0 < \alpha \leq \min\{ \alpha_1,\alpha_2 \}$, and any $t > 0$, we have
\begin{equation}
E(t) \leq 3 E(0) e^{-\frac{2\alpha}{3} t}.
\end{equation}
\end{theorem}
\begin{proof}
In light of~(\ref{eq:ode}),~(\ref{eq:Bbelow}), and the lower bound
in~(\ref{eq:AequivE}), we have that
\begin{equation}
A^\prime(t)  + \frac{2\alpha}{3} A(t) \leq 0,
\end{equation}
so that
\begin{equation}
A(t) \leq A(0) e^{-\frac{2\alpha}{3} t}.
\end{equation}
Using the upper and lower bounds of $A$ in~(\ref{eq:AequivE}) gives
the desired estimate.
\end{proof}

This result shows that the damping term drives an unforced system to
one with a steady, solenoidal velocity field, in which the Coriolis
force balances the pressure gradient term, \emph{i.e.} in a state of
geostrophic balance.  Using the second equation
in~(\ref{eq:discrete_mixed}), we also know that the linearized height
disturbance is steady in time in this case.  These facts together lead
to an elliptic equation for the steady state
\begin{equation}
\label{eq:elliptic}
\begin{split}
\left( \frac{C}{H} u_h , v_h \right) + \frac{1}{\epsilon} \left(
\frac{f}{H} u_h^\perp , v_h \right) - \frac{\beta}{\epsilon^2} \left(
\eta_h , \nabla \cdot v_h \right) & = 0 \\
\left( \nabla \cdot u_h, w_h \right) & = 0
\end{split}
\end{equation}
It is easy to see that this problem is coercive on the divergence-free
subspaces and thus is well-posed.  Hence, with zero
forcing, both $u_h$ and $\eta_h$ equal zero is the only solution.  The 
zero-energy steady state then cannot have a nonzero solenoidal part.
Moreover, the exponentially decay of $\| u_t \|$ toward zero forces
$u$ to reach its steady state quickly, driving both $u^D$ and $u^S$
toward zero at an exponential rate.  Finally, since 
$\eta_t = -\nabla \cdot u$ almost everywhere, the exponential damping of 
$\| \nabla \cdot u \|$ also forces $\eta$ toward its zero steady state at
the same rate.

Now, we turn to the case where the forcing term is nonzero, adapting this
damping result to give long-time stability.
The same techniques as before now lead to
\begin{equation}
A^\prime(t) + B(t) = \left( \widetilde{F} , u_{h,t} + \alpha u_h^D \right).
\label{eq:ode2}
\end{equation}

\begin{theorem}
For any $0 < \alpha \leq \min\{ \alpha_1,\alpha_2\}$ and
\label{thm:dampedstable}
\begin{equation}
K_\alpha \equiv \frac{1}{2} \left[1 +  \frac{\alpha^2 C_P^2\epsilon^2}{\beta H_*^2} \right],
\end{equation} 
we have the bound
\begin{equation}
E(t) \leq 3 e^{-\frac{\alpha}{3}t} E(0) + \frac{K_\alpha}{\alpha} \int_0^t e^{\frac{\alpha}{3} \left( s - t \right)} \weightednorm{\widetilde{F}}{H}^2 ds.
\end{equation}
\end{theorem}
\begin{proof}
We bound the right-hand side of~(\ref{eq:ode2}) by
\begin{equation}
\begin{split}
 \left( \widetilde{F}  , u_{h,t} + \alpha u_h^D \right) &
\leq \weightednorm{ \widetilde{F}}{H} \weightednorm{ u_{h,t} }{\frac{1}{H}}
+ \alpha C_P \weightednorm{ \widetilde{F} }{H} \weightednorm{ \nabla \cdot u_h^D }{\frac{1}{H}} \\
& \leq 
\left[ \frac{H^*}{2 \delta_1} +  \frac{\alpha C_P}{2 \delta_2} 
\right] \weightednorm{\widetilde{F}}{H}^2
+ \frac{\delta_1}{2} \weightednorm{u_{h,t}}{\frac{1}{H}}
+ \frac{\alpha C_P \delta_2}{2 H_*}\norm{ \nabla \cdot u_h^D }^2
\end{split}
\end{equation}
We put $\delta_2 =
\frac{\beta\delta_1 H_*}{\alpha C_P \epsilon^2}$ to find
\begin{equation}
 \left( \widetilde{F}  , u_{h,t} + \alpha u_h^D \right)
\leq
\frac{1}{\delta_1} K_\alpha
\weightednorm{\widetilde{F}}{H}^2
+ \delta_1 E(t).
\end{equation}
This turns~(\ref{eq:ode2}) into the differential inequality
\begin{equation}
A^\prime(t) + B(t) \leq \frac{K_\alpha}{\delta_1}
\weightednorm{\widetilde{F}}{H}^2
+ \delta_1 E(t).
\end{equation} 
Using ~(\ref{eq:Bbelow}), we obtain 
\begin{equation}
A^\prime(t) + \alpha E(t) \leq  \frac{K_\alpha}{\delta_1}
\weightednorm{\widetilde{F}}{H}^2
+ \delta_1 E(t).
\end{equation}
At this point, we specify $\delta_1 = \frac{\alpha}{2}$ so that,
with~(\ref{eq:AequivE}) we have
\begin{equation}
A^\prime(t) + \frac{\alpha}{3} A(t) 
\leq  \frac{K_\alpha}{\alpha} 
\weightednorm{\widetilde{F}}{H}^2.
\end{equation}
This leads to the bound on $A(t)$
\begin{equation}
A(t) \leq e^{-\frac{\alpha}{3} t} A(0)
+ \frac{K_\alpha}{\alpha} \int_0^t e^{\frac{\alpha}{3}\left(s-t\right)} \weightednorm{\widetilde{F}}{H}^2 ds.
\end{equation}
Using~(\ref{eq:AequivE}) again gives the desired result.
\end{proof}

These stability results have important implications for tidal computations.
Theorem~\ref{thm:dampedstable} shows long-time stability of the
system.  Our stability result  also shows that the semidiscrete method
captures the three-way geotryptic balance between Coriolis, pressure
gradients, and forcing.  Moreover, we also can demonstrate that
``spin-up'', the process by which in practice tide models are started
from an arbitrary initial condition and run until they approach their
long-term behavior, is justified for this method.  To see this, the
difference between any two solutions with equal forcing but differing initial 
conditions will satisfy the same~\eqref{eq:secondorderdiscrete} with
nonzero initial conditions and zero forcing.  Consequently, the
difference must approach zero exponentially fast.  This means that we
can define a global attracting solution in the standard way (that is,
take $\eta(x,t;t^*)$, $u(x,t;t^*)$ for $0 > t^*$ and $t > t^*$ as
the solution starting from zero initial conditions at $t^*$ and define
the global attracting solution as the limit as $t^* \rightarrow
-\infty$), to which the solution for any condition becomes
exponentially close in finite time.  The error estimates we
demonstrate in the next section then can be used to show that the
semidiscrete finite element 
solution for given initial conditions approximates this global
attracting solution arbitrarily well by picking $t$ large enough that
the difference between the exact solution with those initial
conditions and the global attracting
solution is small and then letting $h$ be small enough that the
finite element solution approximates that exact solution well.

\section{Error estimates}
\label{se:error}
Optimal \emph{a priori} error estimates follow by applying our
stability estimates to a discrete equation for the difference between
the numerical solution and a projection of the true solution.
We define 
\begin{equation}
\begin{split}
\chi &\equiv \Pi u - u, \\
\rho & \equiv \pi \eta - \eta, \\
\theta_h & \equiv \Pi u - u_h,  \\
\zeta_h & \equiv \pi \eta - \eta_h.
\end{split}
\end{equation}

The projections $\Pi u$ and $\pi \eta$ satisfy the first-order system
\begin{equation}
\begin{split}
\left( \frac{1}{H} \Pi u_{t} , v_h \right) 
+ \frac{1}{\epsilon} \left( \frac{f}{H} \Pi u^\perp , v_h \right) 
- \frac{\beta}{\epsilon^2} \left( \pi \eta ,
\nabla \cdot v_h \right) + \left( \frac{C}{H} \Pi u , v_h \right) & =
\left( F + \frac{f}{\epsilon H} \chi + \frac{1}{H}\chi_t +\frac{C}{H} \chi , v_h \right)
, \\
\left( \pi \eta_{t} , w_h \right) + \left( \nabla \cdot \Pi u  , w_h \right)& = 0.
\end{split}
\end{equation}
Subtracting the discrete equation~(\ref{eq:discrete_mixed}) from this gives
\begin{equation}
\label{eq:firstordererror}
\begin{split}
\left( \frac{1}{H} \theta_{h,t} , v_h \right) 
+ \frac{1}{\epsilon} \left( \frac{f}{H} \theta_h^\perp , v_h \right) 
- \frac{\beta}{\epsilon^2} \left( \zeta_h ,
\nabla \cdot v_h \right) + \left( \frac{C}{H} \theta_h , v_h \right) & =
\left( \frac{f}{\epsilon H} \chi +  \frac{1}{H} \chi_t + \frac{C}{H} \chi , w_h \right)
, \\
\left( \zeta_{h,t} , w_h \right) + \left( \nabla \cdot \theta_h  , w_h \right)& = 0.
\end{split}
\end{equation}

By choosing the initial conditions for the discrete problem as
$u_h(\cdot,0) = \Pi u_0$ and
$\eta_h(\cdot,0) = \pi \eta_0$, the initial conditions for these error equations are
\begin{equation}
\begin{split}
\theta_h(\cdot,0) & = 0, \\
\eta_h(\cdot,0) & = 0.
\end{split}
\end{equation}

We start with $L^2$ estimates for the height and momentum variables,
based on the stability result for the first order system. 
\begin{proposition}
\label{prop:L2discerr}
For any $t > 0$, provided that $u,u_t \in L^2([0,t],H^{k+\sigma}(\Omega))$, 
\begin{equation}
\begin{split}
& \frac{1}{2} \weightednormattime{\theta_h}{\frac{1}{H}}{t}^2
+ \frac{\beta}{2\epsilon^2} \normattime{\zeta_h}{t}^2 \\
\leq
&
\frac{C_\pi^2 h^{2\left( k + \sigma \right)}}{C_* H_*}
\int_0^t 
\frac{1}{\epsilon}
\weightedseminormattime{u}{k+\sigma}{s}^2
+ \weightedseminormattime{u_t}{k+\sigma}{s}^2
+ C^* \weightedseminormattime{u}{k+\sigma}{s}^2 ds.
\end{split}
\end{equation}
\end{proposition}
\begin{proof}
We apply Proposition~\ref{prop:firstorderstability} to~(\ref{eq:firstordererror}) to find
\begin{equation}
\frac{1}{2} \weightednormattime{\theta_h}{\frac{1}{H}}{t}^2
+ \frac{\beta}{2\epsilon^2} \normattime{\zeta_h}{t}^2
\leq
\frac{1}{2C_*} \int_0^t \weightednormattime{\frac{f}{\epsilon H} \chi +  \frac{1}{H} \chi_t +
  \frac{C}{H} \chi }{H}{s}^2 ds.
\end{equation}

Note that for any $g$,
\[
\weightednorm{\frac{1}{H} g}{H}^2
= \int_\Omega H \left( \frac{1}{H} \left| g \right| \right)^2 dx
= \int_\Omega \frac{1}{H} \left| g \right|^2 dx
= \weightednorm{g}{\frac{1}{H}}^2.
\]
Using this, that $(a+b)^2 \leq 2 \left( a^2 + b^2 \right)$, and norm 
equivalence bounds
the right-hand side above by
\[
\begin{split}
& \frac{1}{C_*H_*} \int_0^t 
\normattime{\frac{f}{\epsilon} \chi}{s}^2 +
\normattime{\chi_t}{s}^2 +
\normattime{C \chi}{s}^2 ds \\
\leq
& \frac{1}{C_*H_*} \int_0^t 
\frac{1}{\epsilon} \normattime{\chi}{s}^2 +
\normattime{\chi_t}{s}^2 +
C^* \normattime{\chi}{s}^2 ds
\end{split}
\]
and the approximation estimate~(\ref{eq:PiL2}) finishes the proof.
\end{proof}

Since
\[
\frac{1}{2} \weightednorm{\left( u - u_h
  \right)}{\frac{1}{H}}^2
+
\frac{\beta}{2 \epsilon^2} \norm{\eta - \eta_h}^2
\leq 
\weightednorm{\rho}{\frac{1}{H}}^2 + 
\weightednorm{\zeta_h}{\frac{1}{H}}^2
+ \frac{\beta}{2 \epsilon^2} \norm{\chi}^2
+ \frac{\beta}{2 \epsilon^2} \norm{\theta}^2,
\]
we combine this result with the approximation estimates to obtain
\begin{theorem}
\label{th:L2 bounds}
If the above hypotheses hold, and also $u \in
L^\infty([0,t];H^{k+\sigma}(\Omega))$ and $\eta \in
L^\infty([0,t];H^k(\Omega))$, we have the error estimate 
\begin{equation}
\begin{split}
\frac{1}{2} \weightednormattime{\left( u - u_h
  \right)}{\frac{1}{H}}{t}^2
+
\frac{\beta}{2 \epsilon^2} \normattime{\left( \eta - \eta_h
  \right)}{t}^2
& \leq \frac{C_\Pi^2 h^{2\left( k + \sigma\right)}}{H_*}
\weightedseminormattime{ u }{ k + \sigma }{ t }^2
+
\frac{C_\pi^2 \beta h^{2k}}{\epsilon^2}
\weightedseminormattime{ \eta }{k}{t}^2 \\
& + 
\frac{2 C_\pi^2 h^{2\left( k + \sigma \right)}}{C_* H_*}
\int_0^t \weightedseminormattime{u_t}{k+\sigma}{s}^2
+ C^* \weightedseminormattime{u}{k+\sigma}{s}^2 ds.
\end{split}
\end{equation}
\end{theorem}

Note that our bound on the error equations in Proposition~\ref{prop:L2discerr}
depend only on the approximation properties of the velocity space,
while the full error in the finite element
solution depends on the approximation properties of both spaces.
Consequently, the velocity approximation using BDM elements is
suboptimal.  Using RT or BDFM elements, both fields are approximated
to optimal order. 

Now, we use our estimates based on the second-order system to obtain
error estimates for the time derivative and divergence of the momentum.
The projection $\Pi u$ satisfies the perturbed equation
\begin{equation}
\begin{split}
& \left( \frac{1}{H} \Pi u_{tt} , v_h \right) 
+ \frac{1}{\epsilon} \left( \frac{f}{H} \Pi u_{t}^\perp , v_h \right) 
+ \frac{\beta}{\epsilon^2} \left( \nabla \cdot \Pi u , \nabla \cdot v_h \right)
+ \left( \frac{C}{H} \Pi u_{t} , v_h \right) \\
= &
\left( \frac{1}{H} \chi_{tt} , v_h \right)
+ \frac{1}{\epsilon} \left( \frac{1}{H} \chi^\perp_t , v_h \right)
+ \left( \frac{C}{H} \chi_t , v_h \right)
+ \left( \widetilde{F}  , v_h \right).
\end{split}
\label{eq:secondorderprojeq}
\end{equation}
As in the first-order case, we have 
$\theta_h \equiv \Pi u - u_h$, and
subtracting~(\ref{eq:secondorderdiscrete}) from
(\ref{eq:secondorderprojeq}) gives
\begin{equation}
\begin{split}
& \left( \frac{1}{H} \theta_{h,tt} , v_h \right) 
+ \frac{1}{\epsilon} \left( \frac{f}{H} \theta_{h,t}^\perp , v_h \right) 
+ \frac{\beta}{\epsilon^2} \left( \nabla \cdot \theta_{h} , \nabla \cdot v_h \right)
+ \left( \frac{C}{H} \theta_{h,t} , v_h \right) \\
= &
\left( \frac{1}{H} \chi_{tt} , v_h \right)
+ \frac{1}{\epsilon} \left( \frac{f}{H} \chi^\perp_t , v_h \right)
+ \left( \frac{C}{H} \chi_t , v_h \right).
\end{split}
\label{eq:thetaeq}
\end{equation}
Theorem~\ref{thm:dampedstable} and approximation
estimates for $\chi$ give this result.
\begin{proposition}
Let $\alpha = \alpha_* = \min\{ \alpha_1 , \alpha_2\}$ and suppose
that $u_t, u_{tt} \in L^1(0,T;H_{k+1})$.  Then
\begin{equation}
\frac{1}{2} \weightednorm{\theta_{h,t}}{\frac{1}{H}}^2
+ \frac{\beta}{2\epsilon^2} \norm{\nabla \cdot \theta_h}^2
\leq \frac{K_{\alpha_*}C_\Pi^2 h^{2\left(k+\sigma\right)}}{\alpha_* H_*} 
\int_0^t e^{\frac{\alpha_*}{3}\left( s - t \right)} 
\left(
\weightedseminorm{u_{tt}}{k+1}^2
+ \left( \frac{1}{\epsilon} + C^* \right)
\weightedseminorm{u_{t}}{k+1}^2
\right).
\end{equation}
\end{proposition}
\begin{proof}
Applying the stability estimate to~(\ref{eq:thetaeq}), noting that
$\theta_h =0$ at $t=0$ gives
\begin{equation}
\frac{1}{2} \weightednorm{\theta_{h,t}}{\frac{1}{H}}^2
+ \frac{\beta}{2\epsilon^2} \norm{\nabla \cdot \theta_h}^2
\leq \frac{K_{\alpha_*}}{\alpha_*}
\int_0^t e^{-\frac{\alpha_*}{3}\left( s - t \right)} 
\left(
\weightednorm{\xi_{tt}}{\frac{1}{H}}^2
+ \left( \frac{1}{\epsilon} + C^* \right)
\weightednorm{\xi_{t}}{\frac{1}{H}}^2
\right),
\end{equation}
and applying the norm equivalence and approximation estimate~(\ref{eq:piL2}) gives the result.
\end{proof}

It is straightforward to get from here to a bound on the error
\begin{equation}
\varepsilon^2 \equiv
\frac{1}{2} \weightednormattime{\left( u_t - u_{h,t}
  \right)}{\frac{1}{H}}{t}^2
+
\frac{\beta}{2 \epsilon^2} \normattime{\nabla \cdot \left( u - u_h
  \right)}{t}^2.
\end{equation}

\begin{theorem}
\label{th:exponential error}
If the above assumptions hold, and also $u_t,u_{tt} \in L^\infty([0,t];H^{k+1}(\Omega))$, then
\begin{equation}
\begin{split}
\varepsilon^2 
& \leq \frac{C_\Pi^2 h^{2\left( k + \sigma\right)}}{H_*}
\weightedseminormattime{ u_t }{ k + \sigma }{ t }^2
+
\frac{C_\pi^2 \beta h^{2k}}{\epsilon^2}
\weightedseminormattime{ u }{k+1}{t}^2 \\
& +
\frac{2K_{\alpha_*}C_\Pi^2 h^{2\left(k+\sigma\right)}}{\alpha_* H_*} 
\int_0^t e^{-\frac{\alpha_*}{3}\left( s - t \right)} 
\left(
\weightedseminorm{u_{tt}}{k+1}^2
+ \left( \frac{1}{\epsilon} + C^* \right)
\weightedseminorm{u_{t}}{k+1}^2
\right).
\end{split}
\end{equation}
\end{theorem}

\section{Numerical results}

In this section we present some numerical experiments that illustrate
the estimates derived in the previous sections. In all cases the
equations are discretized in time using the implicit midpoint rule.
The domain is the unit sphere, centred on the origin, which is
approximated using triangular elements arranged in an icosahedral mesh
structure (see Appendix \ref{ap:bendy} for extensions of the results
of this paper to embedded surfaces such as the sphere). All numerical
results are obtained using the open source finite element library,
Firedrake (\url{http://www.firedrake.org}).

First, we verify the energy behavior in the absence of dissipation,
\emph{i.e.} $C=0$. The variables were initialized with ${u}=0$ and
$\eta=xyz$, and the equations were solved with parameters
$\epsilon=\beta=0.1$, $f=1$, $H=1 + 0.1\exp(-x^2)$, and $\Delta t=0.01$.
The energy is conserved by the continuous-time spatial
semi-discretization, and is quadratic. Since the implicit midpoint
rule time-discretization preserves all quadratic invariants (see \cite{LeRe2004}, for example), we expect exact energy conservation in
this case; this was indeed observed as shown in Figure
\ref{fig:energy}. Upon introducing a positive dissipation constant
$C=0.1$, we observe both that the energy is monotonically decreasing
(as implied by Proposition \ref{prop:monotonic}), and is scaling
exponentially in time (as implied by Theorem \ref{th:exp}). These
results are also illustrated in Figure \ref{fig:energy}.

\begin{figure}
\centerline{\includegraphics[width=8cm]{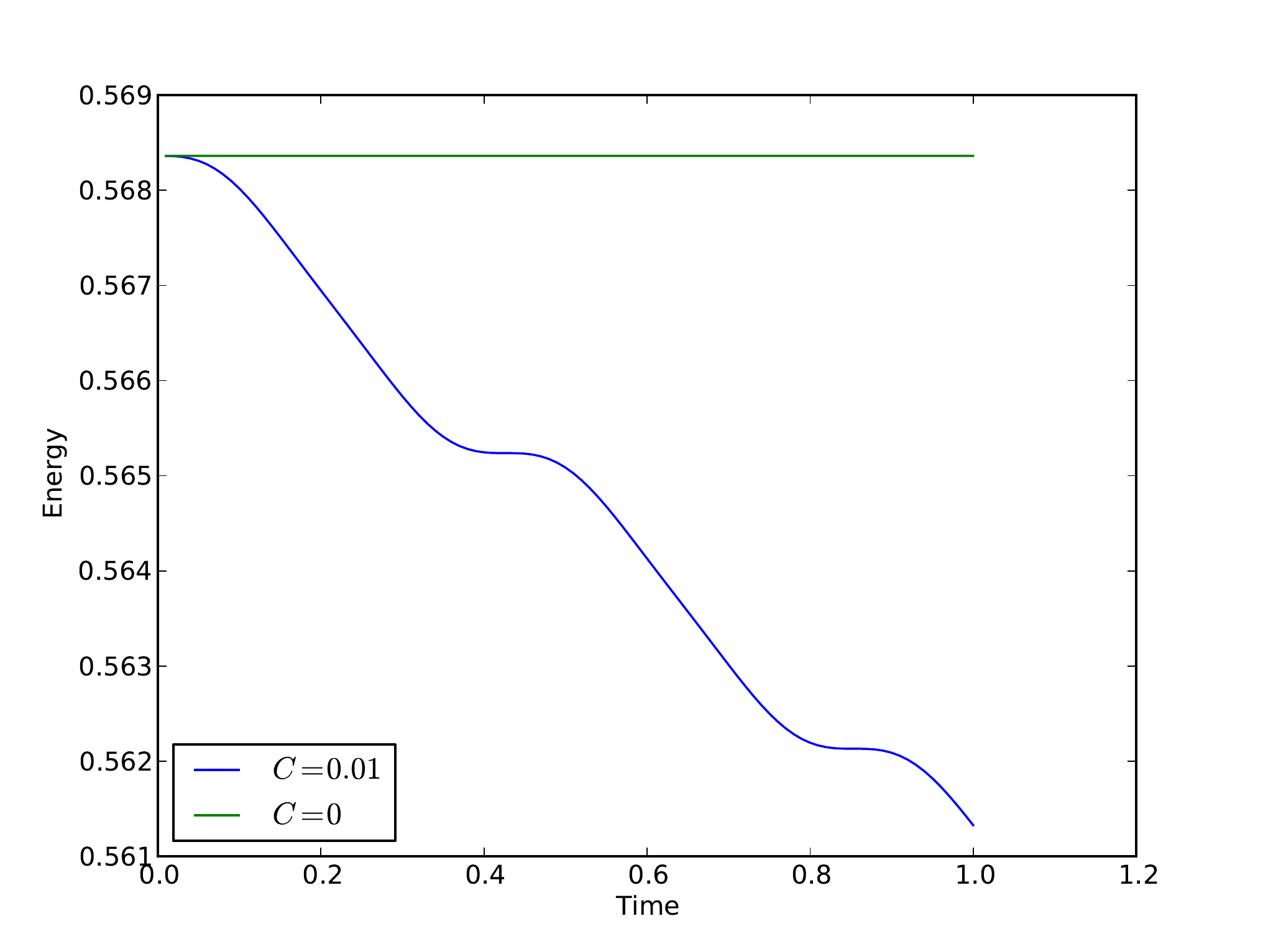}
\includegraphics[width=8cm]{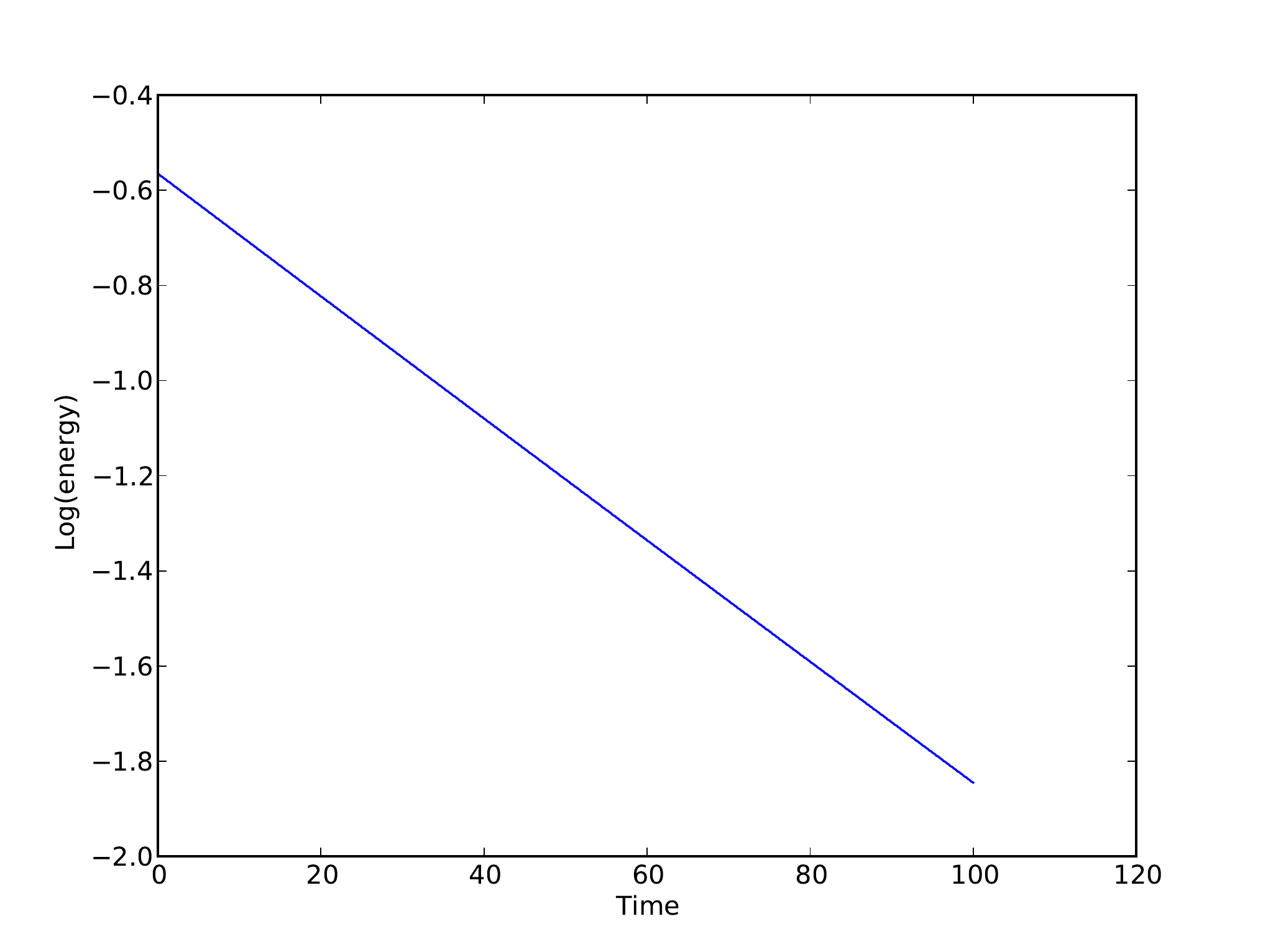}}
\caption{\label{fig:energy}Plots of the evolution of energy with time
  in the cases $C=0$ and $C=0.01$. {\bfseries Left}: Energy-time plots
  for $C=0$ and $C=0.01$, over the time interval $0<t<1$. For $C=0$ we
  observe exact energy conservation as expected. For $C=0.01$ the
  energy is monotonically decreasing as expected. {\bfseries Right}:
  Energy-time plot for $C=0.01$ on a logarithmic scale over the time
  interval $0<t<50$. Then energy is decaying exponentially in time, as
  expected.}
\end{figure}

Second, we verify the convergence results proved in Section
\ref{se:error}. This was done by constructing a reference solution
using the method of manufactured solutions, \emph{i.e.} by choosing
the solution
\[
u = \cos(\Omega t)\left(
-\frac{1}{12}(yz(1 - 3x^2),
-\frac{1}{12}(xz(1 - 3y^2),
-\frac{1}{12}(xy(1 - 3z^2)
\right)
, \quad \eta = -\sin(\Omega t)\frac{xyz}{12},
\]
where we have expressed the velocity in three dimensional coordinates
even though it is constrained to remain tangential to the sphere. Here
$\eta$ and $u$ are chosen to solve the continuity equation for $\eta$
exactly, and $F$ is then chosen so that the $u$ equation is satisfied.
We used the parameters $\epsilon=\beta=0.1$, $f=H=1$, $C=1000$,
$\Omega=2$, and chose $\Delta t=10^{-5}$ in order to isolate the error
due to spatial discretization only. We ran the solutions until $t=0.3$
and computed the time-averaged $L^2$ error for $\eta$. Plots are shown in
Figure \ref{fig:MMS}; they confirm the expected first order
convergence rate for $V=$RT0, $Q=$DG0, and the expected second order
convergence rate for $V=$RT1, $Q=$DG1.

\begin{figure}
\centerline{\includegraphics[width=8cm]{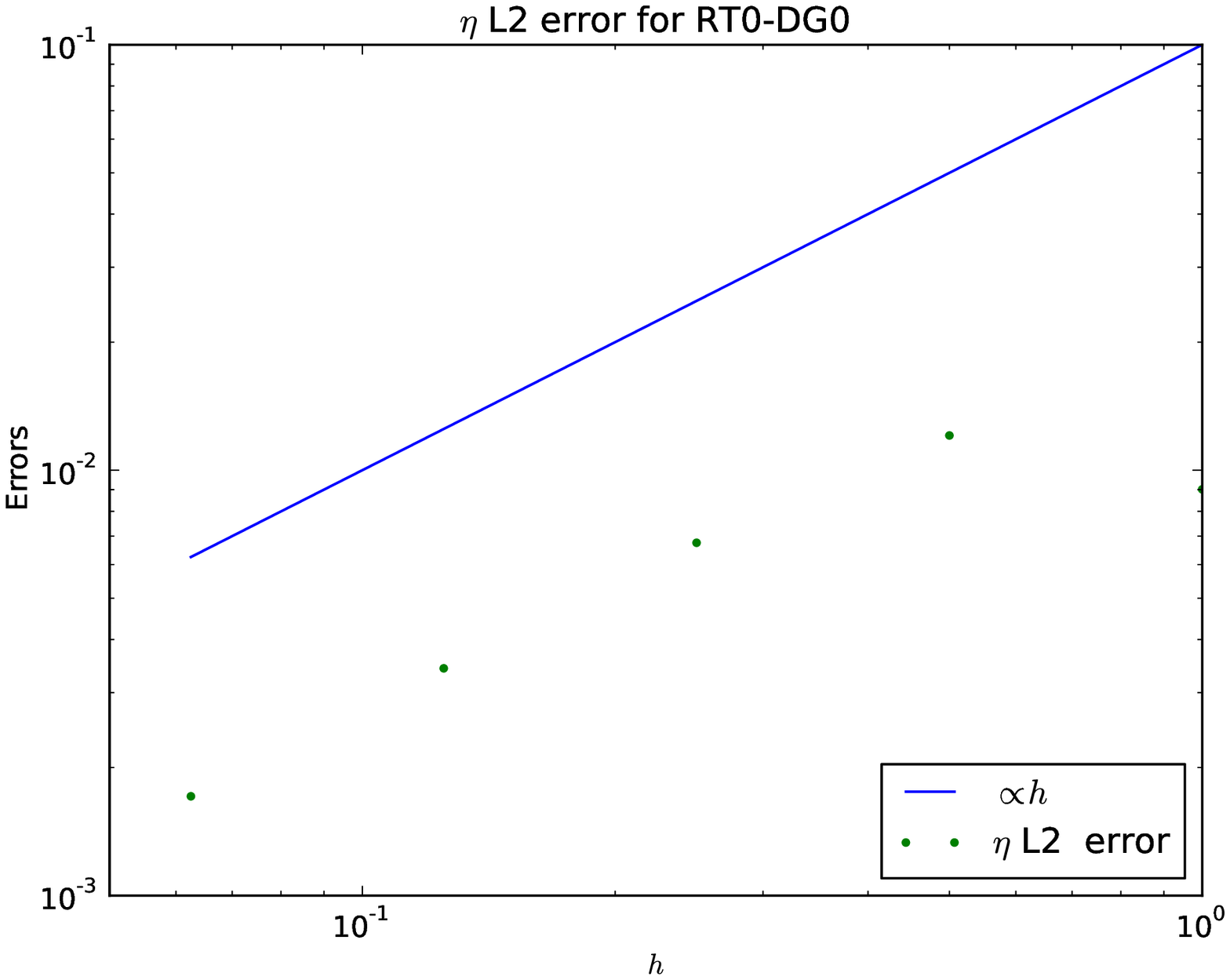}
\includegraphics[width=8cm]{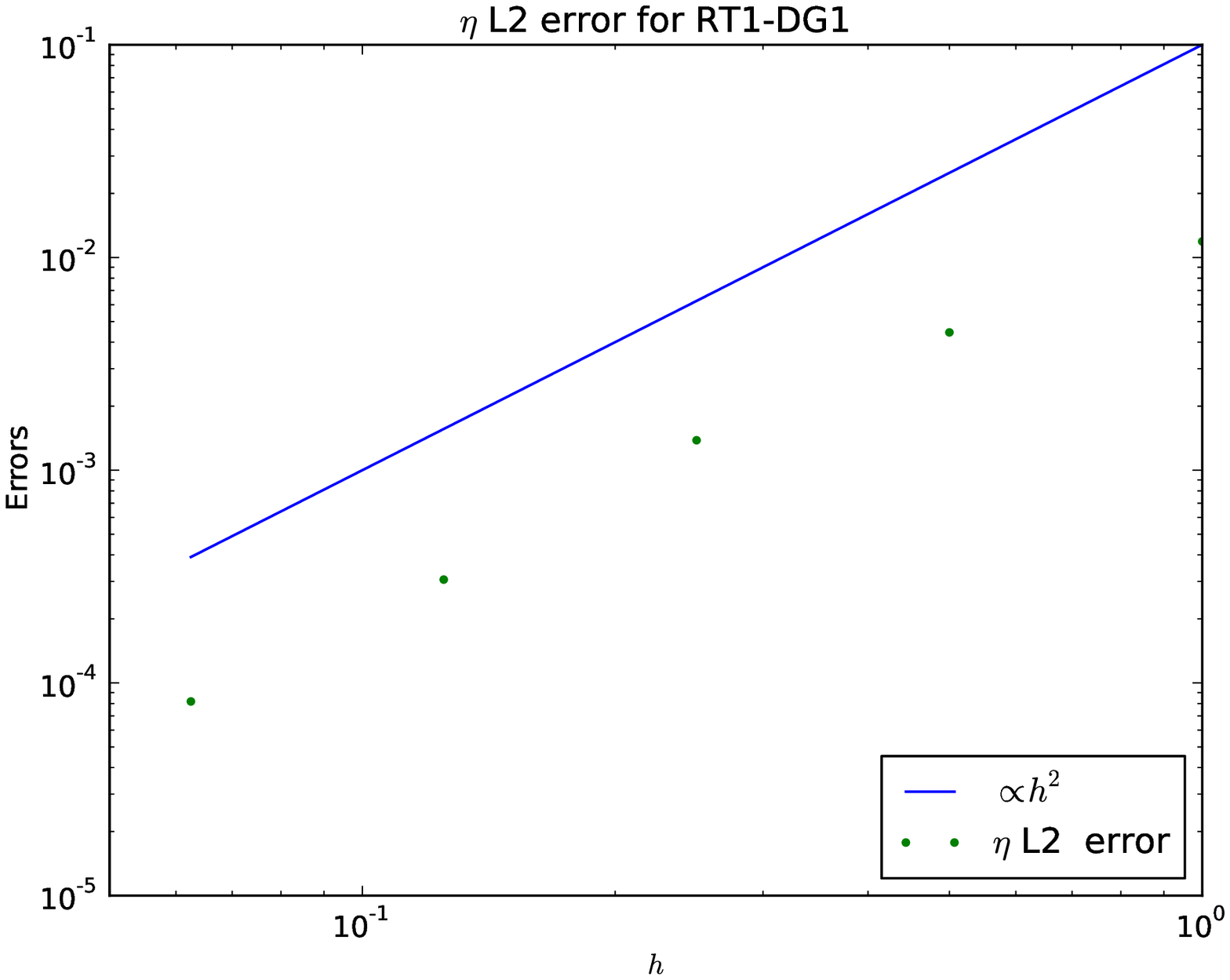}}
\caption{\label{fig:MMS}Convergence plots obtained from the method of
  manufactured solutions, showing the time-integrated $L^2$ error in
  $\eta$ against the typical element edge length $h$. {\bfseries Left:}
  Plot for RT0-DG0, the error is proportional to $h$ as expected. {\bfseries Right:} Plot for RT1-DG1, the error is proportional to $h^2$ as expected.}
\end{figure}

Finally, we illustrate that this type of discretization excludes the
possibility of spurious solutions. In the case of the linear
forced-dissipative tidal equations with time-dependent forcing, the
continuous equations have the property that the solutions lose memory
of the initial conditions exponentially quickly with timescale
determined from $C$ and the other parameters (and bounded by $\alpha$
in Theorem \ref{th:exp}).  As discussed among our stability results, any two solutions with 
different initial conditions should converge to the same solution as
$t\to \infty$. We illustrate this by randomly generating initial conditions for two
solutions $(u_1,\eta_1)$ and $(u_2,\eta_2)$ with the same time-periodic forcing,
\[
(F,v) = \frac{\beta}{\epsilon^2} \sin(t) (xyz,\nabla\cdot v), 
\quad \forall v \in V,
\]
 and measuring the
difference between them as $t\to \infty$. In performing this test,
care must be taken to ensure that $\eta_1$ and $\eta_2$ both have zero
mean as implied by the perturbative derivation of the linear equations
(since the dissipation cannot influence the mean component). In this
experiment, we used the parameters $\epsilon =\beta = 0.1$, $C =
10.0$, $\Delta t=0.01$ and we used an icosahedral mesh of the sphere
at the fourth level of refinement. We indeed observed that the two
solutions converge to each other exponentially quickly in the $L^2$ norm,
as illustrated in Figure \ref{fig:sync}.

\begin{figure}
\centerline{\includegraphics[width=10cm]{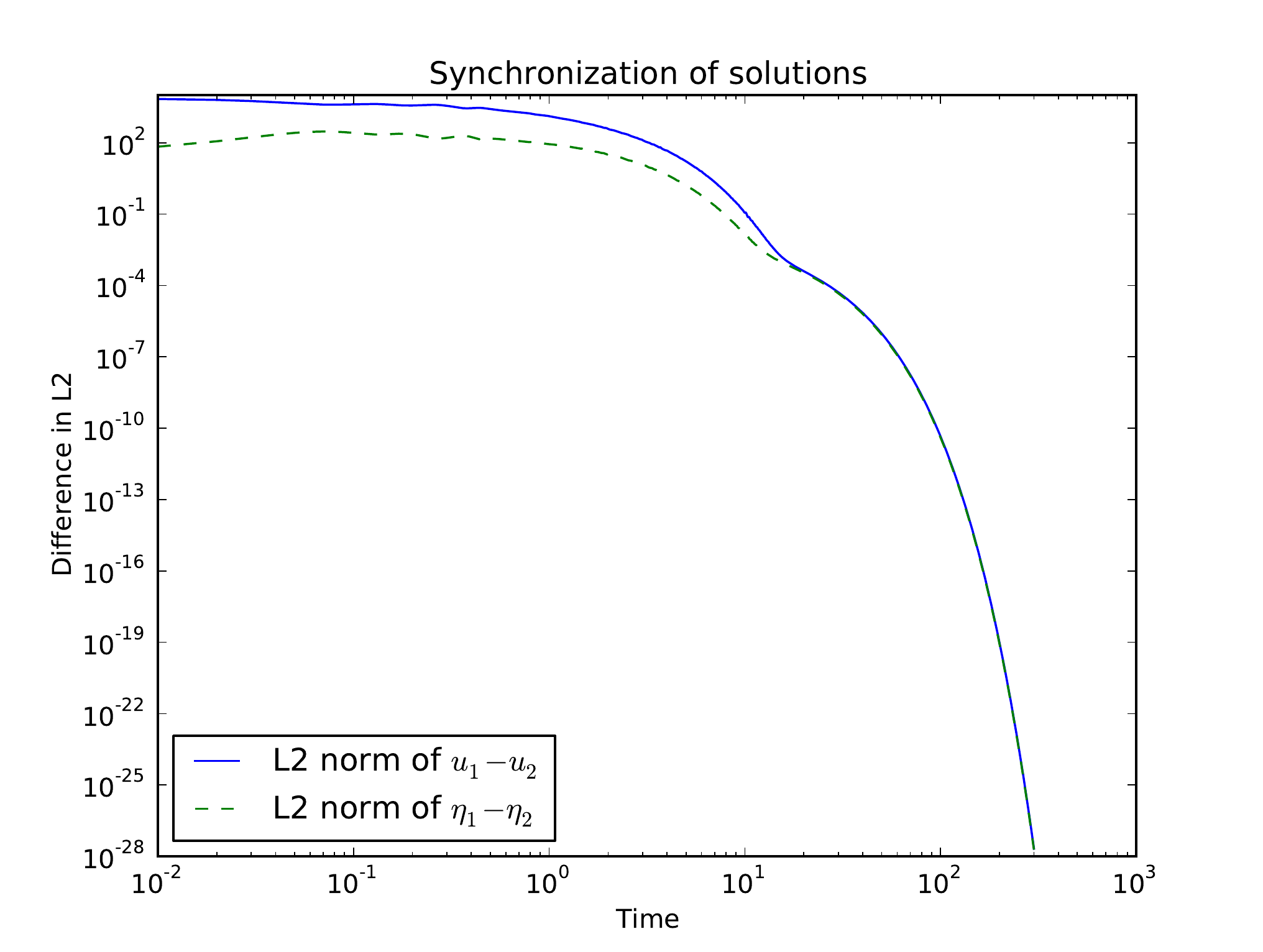}}
\caption{\label{fig:sync}Plot of the $L^2$ difference between two pairs
  of solutions $(u_1,\eta_1)$ and $(u_2,\eta_2)$ with different
  randomly generated initial conditions but the same forcing, as a
  function of time. As expected, the difference converges to zero
  (eventually with exponential rate) as $t\to \infty$, demonstrating
  the absence of spurious solutions.}
\end{figure}

\section{Conclusions and future work}
We have presented and analyzed mixed finite element methods for the
linearized rotating shallow equations with forcing and linear drag
terms.  Our more delicate energy estimates rely on an equivalence
between the first order form and a second order form, and this
equivalence itself relies on fundamental properties of classical
$H(\mathrm{div})$ finite elements.  In particular, our estimates show
that the mixed spatial discretization accurately captures the
long-term energy of the system, in which damping balances out forcing
to prevent energy accumulation.  Because of the linearity of the
problem, our energy estimates also give rise to \emph{a priori}
error estimates that are optimal for Raviart-Thomas and
Brezzi-Douglas-Fortin-Marini elements.  Numerical results confirm
both the stability and convergence theory given.

In the future, we hope to extend this work in several directions.
First, we hope to study the more realistic quadratic damping model,
which will require new techniques to handle the nonlinearity.  Second,
our estimates have only handled the semidiscrete case, and it is
well-known that time-stepping schemes do not always preserve the right
energy balances.  Without damping or forcing, the implicit midpoint method
preserves exact energy balance, and a symplectic Euler method will
exactly conserve an approximate functional for linear problems.  It remains to be seen how to give a
rigorous fully discrete analysis, either including damping by a
fractional step or fully implicit method.  Finally, even explicit or
symplectic time-stepping will require us to consider linear algebraic
problems, as it is typically not possible to perform mass lumping for
$H(\mathrm{div})$ spaces on triangular meshes.  Implicit methods will
require additional care.

\appendix
\section{Extension to the sphere and other embedded manifolds}
\label{ap:bendy}
Global tidal simulations are performed in spherical geometry, so it is
necessary to consider mixed finite element discretization using meshes
of isoparametric elements that approximate the sphere. This
constitutes a variational crime since the domain $M_h$ supporting the
mesh is only the same as the manifold $M$ in the limit $h\to
0$. Recently, the topic of mixed finite elements on embedded manifolds
was comprehensively analyzed by \cite{holst2012geometric}, following
previous work on nodal finite elements. Here, we sketch out how to use
their approach to extend the results of this paper to embedded
manifolds.

In the case of curved domains such as the surface of the sphere,
$\hdiv$ elements are implemented \emph{via} Piola transforms from a
reference triangle. This means that (a) the velocity fields are always
tangential to the mesh element, and (b) normal fluxes $u\cdot n$ take
the same value on each side of element boundaries, as required to
obtain a divergence that is bounded in $L^2$ (an approach to practical
implementation of these finite element spaces on manifolds is provided
by \cite{RoHaCoMc2013}). Similarly, the discontinuous $L^2$ spaces are
implemented using a transformation from the reference triangle
that includes scaling by the determinant of the Jacobian $J_e$; this
ensures that the surface divergence maps from $V_h$ onto $W_h$.

In this case $V_h\not\subset V$, $W_h \not \subset
W$. \cite{holst2012geometric} dealt with this problem by constructing
operators $\iota_{V_h}:V_h \to V$ and $\iota_{W_h}:W_h\to W$ such that
\[
\Pi \circ \iota_{V_h}= \Id_{V_h}, \quad \pi \circ \iota_{W_h} = \Id_{W_h},
\]
where $\Pi$ and $\pi$ are projections from $V$ to $V_h$ and $W$ to
$W_h$ respectively; these two operators commute with $\nabla\cdot$
defined on $M_h$. In particular,
\[
\left(\pi\eta,w_h\right) = \left(\eta,\iota_{W_h}w_h
\right), \quad \forall w_h \in W_h, \, \eta \in W.
\]
The estimates (\ref{eq:PiL2}-\ref{eq:piL2}) then
hold with $\iota_{V_h}\circ \Pi$ and $\iota_{W_h}\circ \pi$ replacing
$\Pi$ and $\pi$ respectively, provided that the polynomial expansion
of the element geometries in $M_h$ have at the same approximation
order as $V_h$ and $W_h$. There is also still a discrete
Poincar\'e-Friedrichs inequality for $V_h$. This means that all of our
stability results \ref{se:energy} hold in the manifold case, and it
remains to deal with the error estimates. This is done by introducing
further variables $u_h'\in V_h$, $\eta'_h\in W_h$ satisfying
\begin{equation}
\begin{split}
\left( \frac{1}{H} \iota_{V_h}u_{h,t}' , \iota_{V_h} v_h \right) 
+ \frac{1}{\epsilon} \left( \frac{f}{H}  (\iota_{V_h}u_h')^\perp , \iota_{V_h} v_h \right) \qquad \qquad & \\
 - \frac{\beta}{\epsilon^2} \left( \iota_{W_h}\eta'_h ,
\iota_{W_h}\nabla \cdot v_h \right) + \left( \frac{C}{H} \iota_{V_h}u'_h , 
\iota_{V_h}v_h \right) & =
\left( F , \iota_{V_h}v_h \right)
, \\
\left( \eta'_{h,t} , w_h \right) + \left( 
\nabla \cdot u_h'  ,w_h \right)& = 0.
\end{split}
\label{eq:discrete_mixed_J}
\end{equation}
This equation is of the form \eqref{eq:discrete_mixed} but with
a modified inner product on $V_h$. Therefore, all of our stability estimates
also hold for this modified equation. 

We split the error in $u$ and $\eta$ by writing
\begin{equation}
\begin{split}
u - \iota_{V_h}u_h &= -\chi + \iota_{V_h}\theta'_h + \iota_{V_h}\theta_h,\\
\eta - \iota_{W_h}\eta_h &= -\rho + \iota_{W_h}\zeta'_h + \iota_{W_h}\zeta_h, \\
\end{split}
\end{equation}
where
\begin{equation}
\begin{split}
\chi &\equiv \iota_{V_h}\Pi u - u, \\
\rho & \equiv \iota_{W_h}\pi \eta - \eta, \\
\theta'_h & \equiv \Pi u -  u_h',  \\
\zeta'_h & \equiv \pi \eta_ - \eta_h'.\\
\theta_h & \equiv u_h' - u_h,  \\
\zeta_h & \equiv \eta_h' - \eta_h.
\end{split}
\end{equation}
We can bound $\theta'_h$ and $\zeta'_h$ by applying
Proposition~\ref{prop:firstorderstability} adapted to Equation
\eqref{eq:discrete_mixed_J}, \emph{i.e.} by substituting
$v=\iota_{V_h}v_h$ into \eqref{eq:mixed} and rearranging so that it
takes the form of \eqref{eq:discrete_mixed_J} with a forcing defined
in terms of $u$, then subtracting
\eqref{eq:discrete_mixed_J}. Similarly, $\theta_h$ and $\zeta_h$ may
be bounded by rearranging Equation \eqref{eq:discrete_mixed_J} into
the form of \eqref{eq:mixed}, then subtracting \eqref{eq:mixed}. Terms
appear that are proportional to $\|\Id-J\|$ where 
\[
J_{V_h} = \iota_{V_h}^*\iota_{V_h}, \quad
J_{W_h} = \iota_{W_h}^*\iota_{W_h},
\]
and $\|\Id-J\|$ is the maximum of the operator norms of
$\Id_{V_h}-J_{V_h}$ and $\Id_{W_h}-J_{W_h}$. \cite{holst2012geometric}
showed that $\|\Id-J\|$ converges to zero as $h\to 0$ with rate
determined by the order of polynomial approximation in the
isoparametric mapping. Hence we obtain a manifold version of Theorem
\ref{th:L2 bounds}, with $u_h$ and $\eta_h$ substituted by
$\iota_{V_h}u_h$ and $\iota_{W_h}\eta_h$ respectively. Similar techniques
lead to a manifold version of Theorem \ref{th:exponential error}.

\bibliographystyle{siam} 
\bibliography{siam}

\end{document}